\documentclass[leqno,12pt]{article} %leqno is the option to put formula numbers on the left side
\usepackage{amsmath, amssymb}
\usepackage{amsthm} %theorem environment option
\theoremstyle{plain} %text of this environment is typesetted in italics
\newtheorem{theorem}{\indent\sc Theorem}[section]
\newtheorem{lemma}[theorem]{\indent\sc Lemma}
\newtheorem{corollary}[theorem]{\indent\sc Corollary}

\theoremstyle{definition} %text of this environment is typesetted in roman letters
\newtheorem{definition}[theorem]{\indent\bf Definition}
\newtheorem{remark}[theorem]{\indent\bf Remark}

\newcommand{\dist}{{\mathrm{d}}}
\newcommand{\LSC}{{\mathrm{LSC}}}
\newcommand{\Ric}{{\mathrm{Ric}}}
\newcommand{\rr}{{\mathbb{R}}}
\newcommand{\USC}{{\mathrm{USC}}}

\begin{document}

\title{Maximum principle for viscosity solutions on Riemannian manifolds
}
\author{ Shige Peng\thanks{Partially  supported by The National Basic Research
Program of China (973 Program) grant No. 2007CB814900.} and Detang
Zhou\thanks{Partially supported by CNPq and FAPERJ of Brazil.}}
\maketitle
\begin {quote}
{{\bf Abstract. } In this paper we consider viscosity solutions to
second order partial differential equations on Riemannian manifolds.
 We prove maximum principles for solutions to Dirichlet problem on a compact Riemannian manifold with boundary.  Using a different method, we generalize maximum principles of
 Omori and Yau to a viscosity version. We also prove maximum principle for parabolic equations.
 }

{ {\it Key words}: viscosity solution, maximum principle, Riemannian
manifold}
\end{quote}

\section{Introduction}The theory of viscosity solutions on $\rr^{n}$ has been an
important area in analysis since the concept was introduced in the
early 1980's by Michael Crandall and Pierre-Louis Lions. As a
generalization of the classical concept of what is meant by a
``solution'' to a partial differential equation, it has been found
that the viscosity solution is the natural solution concept in many
applications of PDE's, including for example first order equations
arising in optimal control (the Hamilton-Jacobi-Bellman equation),
differential games (the Isaacs equation) or front evolution
problems, as well as second-order equations such as the ones arising
in stochastic optimal control or stochastic differential games(see
\cite{cil} and references therein).  It is a natural question to ask
how to generalize the theory to problems on Riemannian manifolds. We
are also motivated by recent works of the first author (\cite{P1},
\cite{P2}) on G-Brownian motion and related stochastic calculus in
$\rr^n$. In section 3 of \cite{P1}, a $G$-normal distribution is
defined to be  a nonlinear expectation given by the viscosity
solution of the following nonlinear parabolic partial differential
equation:
$$ \frac{\partial u}{\partial t}-G(D^2u)=0, \quad u(0,x)=\phi(x), \quad (t,x)\in[0,+\infty)\times \rr^n.$$
Therefore setting up  $G$-Brownian motion and related stochastic
calculus  on a complete noncompact Riemannian manifold needs  the
existence of viscosity solutions of the following nonlinear
parabolic partial differential equation. So we are led to study the
maximum principles.

 Some
special cases have been discussed in comparison theory for
Riemannian distance function and reduced distance function( see
section 9.4 and 9.5 in \cite{ccg}). Up to now, little is known for
general second order partial differential equations in the
references.  Azagra, Ferrera and Sanz \cite{afs} published a paper
in which among the other results they obtained a Hessian estimate
for distance functions and generalized some of results in \cite{cil}
to compact Riemannian manifold with some curvature conditions.
Recently, Harvey and Lawson\cite{HL} studied the Dirichlet problem
for fully nonlinear second order equations on a compact Riemannian
manifold with boundary via closed subsets of the 2-jet bundle and
estabilished existence and uniqueness theorems.

We will study the maxumum principles for viscosity solutions to
second order partial differential equation of the form
\[F(x, u, Du, D^2u)=0
\]
where $u:M\to \mathbb{R}$ is a function and $M$ is a compact
Riemannian manifold with boundary or a complete Riemannian manifold.
It is straightforward to generalize the comcepts of viscosity
subsolution and supersolution to any Riemannian manifold. As is well
known the classical maximum principle  for Dirichlet problem on the
domains in Euclidean space can be easily generalized  to any compact
Riemannian manifold without restrictions on curvature. If we follow
 the  method  used in \cite{cil} we   see that even for
the hyperbolic space  the proof does not go directly.
The reason is that on $\mathbb{R}^n$, one need to compute the square
of distance function $\frac12|x-y|^2$  and as a function on
$\mathbb{R}^n\times \mathbb{R}^n$, its Hessian is
\begin{equation*}
    \begin{pmatrix}
      I_n & -I_n \\
      -I_n & I_n \\
    \end{pmatrix}.
\end{equation*}
As for the Riemannian manifold, the square of distance function
$\frac12\dist(x,y)^2$ is much more complicated. It may be
non-differentiable at some points and   on  hyperbolic space
$\mathbb{H}^n(-1)$ of constant curvature $-1$, it is smooth and its
Hessian can be written as
\begin{equation*}
    \left(
       \begin{array}{cccc}
         1&&-1&\\
         &I_{n-1}\dist\coth \dist & &-\frac{\dist}{\sinh \dist} I_{n-1}\\
         -1&&1&\\
         &-\frac{\dist}{\sinh \dist}I_{n-1} && I_{n-1}\dist\coth \dist \\
       \end{array}
     \right),
\end{equation*}
where $I_{n-1}=\textrm{diag}(1,1,\cdots, 1)$ is $(n-1)\times (n-1)$
unit matrix. We will use $I$ instead of $I_n$ when the dimension is
obvious.

When $M$ is compact Riemannian manifold with boundary, we consider
the Dirichlet problem.  From the proof in section 3 in \cite{cil} we
can deal with  problems in a sufficiently small neighborhood of the
limit point where the function $\frac12\dist(x,y)^2$ is smooth. But
to prove the comparison theorem we  need to estimate the Hessian of
$\frac12\dist(x,y)^2$. In  section \ref{secComp}, we prove a sharp
Hessian estimate for $\frac12\dist(x,y)^2$ for manifold with
sectional curvature bounded below by a constant and a sharp estimate
of Laplacian type for $\frac12\dist(x,y)^2$ for manifolds with Ricci
curvature bounded below a constant. According to these estimate we
 need to modify the proof in \cite{cil} to obtain the desired
 comparison theorem which does not require a curvature restriction.
 More precisely, we prove

\begin{theorem} (see Theorem \ref{thmDiri})
Let $M$  be a  compact Riemannian manifold with or without  boundary
$\partial M$, $F\in C(\mathcal{F}(M),\rr)$  a proper function
satisfying
\begin{equation}
    \beta(r-s)\le F(x,r,p,X)-F(x,s,p,X) \textrm{ for } r\ge s,
\end{equation}
for some positive constant $\beta$ and the following condition (H):
\begin{quote}
there exists a function $\omega:[0,\infty]\to [0,\infty]$ satisfying $\omega(t)>0$ when $ t>0,$ and
$\omega(0+)=0$ such that
\begin{equation}
F(y,r, \delta\iota(\gamma'(0)),X_{1})-F(x,r,
\delta\iota(\gamma'(l)),X_{2})\le
\omega(\delta\dist(x,y)^2+\dist(x,y))
\end{equation}
 for $ X_1\in \mathcal{S}^2
T_{{x}}^*M$ and $X_2\in \mathcal{S}^2T_{{y}}^*M$,  satisfying
$X_1\ge X_2\circ P_{\gamma}(l)$. Here $\iota$ is the dual map
between the tangent and cotangent bundles and $\delta$ is a positive
constant. Here $\mathcal{S}^2T^*M$ is the bundle of symmetric
covariant tensors over $M$.
\end{quote}

 Let $u_1\in \USC(\bar
{M})$ and $u_2\in \LSC(\bar {M})$ be a subsolution and supersolution
of $F=0$ respectively.  Then $u_1-u_2$ cannot achieve a positive
local maximum at any interior point. In particular if $M$  is a
compact Riemannian manifold with boundary $\partial M$ and if
$u_1\le u_2$ on $\partial M$, then $u_1\le u_2$ on $\bar{ M}$.

\end{theorem}

When $M$ is a complete Riemannian manifold without boundary. We have not found
a similar result in the Euclidean case. The new obstruction is that
 we may not have a limit point when $M$ is noncompact. We first
generalize the maximum principle of Omori and Yau (see \cite{O} and
\cite{Y}) to a viscosity-type. As is known, Yau's maximum principle
( \cite{Y}) is stated as the following:
\begin{theorem}Let $M$ be a complete Riemannian manifold with Ricci
curvature bounded below by a constant. Let  $u: M\to \rr$ is a $C^2$
function
 with $\inf u> -\infty$, then for any $\varepsilon>0$, there
 exists a point $x_{\varepsilon}\in M$ such that
 \begin{equation*}
    \begin{split}
         & u(x_\varepsilon)<\inf u+\varepsilon, \\
         & |\nabla u|(x_\varepsilon)<\varepsilon,\\
         & \Delta u(x_\varepsilon)>-\varepsilon.
     \end{split}
 \end{equation*}
\end{theorem}
Its proof uses the gradient estimate of the distance function to a
fixed point (see \cite{CY}) and  for a viscosity solution  the function may be non-differentiable and we cannot use the
gradient estimate. Besides overcoming the difficulties appeared in
compact cases, we used a new penalty function to the corresponding
maximum principles of Omori and Yau for viscosity solutions. Even
for $C^2$ function, we provide a new proof of Omori and Yau's
theorems. We would like remark that such a viscosity version is also
new for $\mathbb{R}^n$. We prove the following (see section
\ref{secPri} for the definitions  of $\bar{J} ^{2,+}u(x) $ and
$\bar{J} ^{2,-}u(x) $.)
\begin{theorem}(see Theorem \ref{thmOmori})
Let $M$ be a complete Riemannian manifold with sectional curvature
bounded below by a constant $-\kappa^2$. Let $u
\in{\mathrm{USC}}(M)$, $v\in{\mathrm{LSC}}(M)$ be two functions
satisfying
 \begin{equation}\label{}
    \mu_0:=\sup_{x\in M}[u(x)-v(x)]< +\infty.
 \end{equation}
 Assume that $u$ and $v$ are bounded from above and below respectively and there exists a  function $\omega:[0,\infty]\to [0,\infty]$ satisfying $\omega(t)>0$ when $ t>0,$ and
$\omega(0+)=0$  such that
\begin{equation}\label{}
    u(x)-u(y)\le \omega(\dist(x,y)).
\end{equation}
Then for each $\varepsilon>0$, there exist $x_{\varepsilon},
y_{\varepsilon}\in M$, such that
$(p_{\varepsilon},X_{\varepsilon})\in \bar{J}%
^{2,+}u(x_{\varepsilon}),\  \ (q_{\varepsilon},Y_{\varepsilon})\in
\bar{J} ^{2,-}v(y_{\varepsilon}),$ such that
\[
u(x_{\varepsilon})-v(y_{\varepsilon})\geq \mu_{0}-\varepsilon,\
\]
and such that
\[
\dist(x_{\varepsilon},y_{\varepsilon})<\varepsilon,\  \  \
|p_{\varepsilon }-q_{\varepsilon}\circ P_{\gamma}(l)|<\varepsilon,\
\ X_{\varepsilon}\leq Y_{\varepsilon }\circ
P_{\gamma}(l)+\varepsilon  P_{\gamma}(l),\
\]
where $l=\dist(x_{\varepsilon},y_{\varepsilon})$ and $
P_{\gamma}(l)$ is the parallel transport along the shortest geodesic
connecting $x_{\varepsilon}$ and $y_{\varepsilon}$.

\end{theorem}
and
\begin{theorem}(see Theorem \ref{thmYau})
Let $M$ be a complete Riemannian manifold with Ricci curvature
bounded below by a constant $-(n-1)\kappa^2$. Let $u
\in{\mathrm{USC}}(M)$, $v\in{\mathrm{LSC}}(M)$ be two functions
satisfying
 \begin{equation}\label{}
    \mu_0:=\sup_{x\in M}[u(x)-v(x)]< +\infty.
 \end{equation}
 Assume that $u$ and $v$ are bounded from above and below respectively and there exists a  function $\omega:[0,\infty]\to [0,\infty]$ satisfying $\omega(t)>0$ when $ t>0,$ and
$\omega(0+)=0$ such that
\begin{equation}\label{}
    u(x)-u(y)\le \omega(\dist(x,y)).
\end{equation}
Then for each $\varepsilon>0$, there exist $x_{\varepsilon},
y_{\varepsilon}\in M$,
$(p_{\varepsilon},X_{\varepsilon})\in \bar{J}%
^{2,+}u(x_{\varepsilon}),\ $ and $ \
(q_{\varepsilon},Y_{\varepsilon})\in \bar{J}
^{2,-}v(y_{\varepsilon}),$ such that
\[
u(x_{\varepsilon})-v(y_{\varepsilon})\geq \mu_{0}-\varepsilon,\
\]
and such that
\[
\dist(x_{\varepsilon},y_{\varepsilon})<\varepsilon,\  \  \
|p_{\varepsilon }-q_{\varepsilon}\circ P_{\gamma}(l)|<\varepsilon,\
\mathrm{tr} X_{\varepsilon}\leq \mathrm{tr} Y_{\varepsilon
}+\varepsilon,\
\]
where $l=\dist(x_{\varepsilon},y_{\varepsilon})$ and $
P_{\gamma}(l)$ is the parallel transport along the shortest geodesic
connecting $x_{\varepsilon}$ and $y_{\varepsilon}$.

\end{theorem}

One can see that if we take $u$ as a constant function in above
theorems and $v$ is a $C^2$ function we can obtain the estimates
about Hessian and Laplacian respectively. Therefore the above
theorems generalizes Omori and Yau's maximum principles. It is well
known that Yau's theorem has been applied to many geometrical
problems and many functions  in geometry problems are naturally
non-differentiable.

We can follow the text of Users' Guide by Crandall, Ishii and
Lions\cite{cil} to obtain the corresponding results about existence,
uniqueness of partial differential equations. We hope to apply the
results to study more equations as well as  some geometrical
problems in \cite{HL}.

 The rest of this paper is organized as  follows. In section
 \ref{secPri} we give some
 basic notations and definitions for viscosity solutions . In section \ref{secComp}, we prove the comparison
 theorems for distance function. The sections 4 and 5 deal with
 maximum principles for viscosity solutions on compact and complete
 Riemannian manifolds respectively. Finally  we consider the
 parabolic equations in section \ref{secPara}.

{\bf Acknowledgements.}  We wish to  thank Xuehong Zhu for giving us
helpful comments and suggestions. After we posted the first version
of this paper on arXiv in 2008,  Professor H.B. Lawson send us their
preprint \cite{HL} which contains results related to our results
when $M$ has a boundary.  We would like thank him for sending us the
paper.

\section{Preliminaries} \label{secPri}
Let $M$ denote a Riemannian manifold. $TM$ and $T^{*}M$ are tangent
and cotangent bundles respectively. $\mathcal{S}^2T^*M$ is the
bundle of symmetric covariant tensors over $M$. Denote by
$\mathcal{F}(M)$ the bundle product  of $M\times \mathbb{R}$,
$T^{*}M$ and $\mathcal{S}^2 T^{*}M$.

Given a function $f\in C^2(M,\rr)$, the Hessian of $f$ is defined as
\begin{equation*}
    D^2f(X,Y)=\left\langle\nabla_X\nabla
    f,Y\right\rangle=X(Yf)-(\nabla_XY)f
\end{equation*}
where $X,Y$ are vector fields on $M$ and $\nabla$ is the Riemannian
connection on $M$, hence $D^2f\in \mathcal{S}^2T^*M $.  In this
paper, we abuse the notations to write $A=D^2f$ for a matrix $A$
which means $\langle AX, Y\rangle=D^2f(X, Y)$ for all $X,Y\in TM$.
For a differentiable map $\phi$ between two Riemannian manifolds $N$
and $M$ and a function $f\in C^2(M,\rr)$, we have
\begin{equation}\label{eq0.5}
    D^2f\circ\phi(\xi,\eta)=\xi(\eta(f\circ\phi))-(\nabla^N_\xi\eta)(f\circ
    \phi),
\end{equation}
where $\xi,\eta$ are vector fields on $N$ and $\nabla^N$ is the
Riemannian connection on $N$. As a special case, for any fixed point
${x_0}\in M$, we take $N$ as $T_{x_0}M$ and $\phi$ as the
exponential map $\exp: T_{x_0}M\to M$, then in this normal
coordinates the (\ref{eq0.5}) implies, at point $O$,
\begin{equation}\label{eq0.6}
    D^2f\circ\phi(\xi,\eta)=\xi(\eta(f\circ\phi)).
\end{equation}
Besides this, it is well known that the differential of the
exponential map $d\exp$ at $O$ gives an isomorphism between
$T_0(T_{x_0}M)$ and $T_{x_0}M$.

The viscosity solution theory applies to certain partial
differential equations of the form $F(x, u, Du, D^2u)=0$ where
$F:\mathcal{F}(M)\to \mathbb{R}$ is a continuous function. As  in
\cite{cil}, we require $F$ to satisfy some fundamental monotonicity
conditions  called {\it proper} which are made up of the two
conditions
\begin{equation}\label{eq0.9}
F(x,r,p,X)\le F(x,s,p, X) \quad\mathrm{ whenever }\quad r\le s ;
\end{equation}
and
\begin{equation}\label{eq01}
    F(x,r,p,X)\le F(x,r,p, Y) \quad\mathrm{ whenever }\quad Y\le X;
\end{equation}
where  $(x,r,p,X)\in \mathcal{F}(M)$ and $(x,s,p,Y)\in
\mathcal{F}(M)$.

For any function $u:M\to \mathbb{R}$, we define
\begin{definition}
\begin{equation}\label{eq02}
    J^{2,+}u(x_0)=\{(p,X)\in T_{x_0}^{*}M\times\mathcal{S}^2 T_{x_0}^{*}M, \quad\mathrm{satisfies}\quad (\ref{eq03}). \}
\end{equation}
where (\ref{eq03}) is
\begin{equation}\label{eq03}
    u(x)\le
    u(x_0)+p(\exp_{x_0}^{-1}x)+\frac12 X(\exp_{x_0}^{-1}x,\exp_{x_0}^{-1}x)+o(|\exp_{x_0}^{-1}x|^2),
    \quad \mathrm{as}\quad x\to x_0,
\end{equation}
where $\exp_{x_0}: T_{x_0}M\to M$ is the exponential map. And
$J^{2,-}u(x_0)$ is defined as
\begin{equation*}
    J^{2,-}u(x_0)=\{(p,X)\in
    T_{x_0}^{*}M\times\mathcal{S}^2 T_{x_0}^{*}M\textrm{ such that }
    (-p,-X)\in J^{2,+}(-u)(x_0)\}.
\end{equation*}
\end{definition}
\begin{remark}\label{rk1}(\ref{eq03}) is equivalent to that the function
$\bar{u}$  on $T_{x_0}M$ defined by $\bar{u}(y)=u(\exp_{x_0}y)$
satisfies
\begin{equation}\label{eq03.5}
    \bar{u}(y)\le
    \bar{u}(0)+\left\langle p,y\right\rangle+\frac12 X(y,y)+o(|y|^2),
    \quad \mathrm{as}\quad y\to 0.
\end{equation}
So we can identify $T_0(T_{x_0}M)$ with $T_{x_0}M$ such that
$(p,X)\in J^{2,+}u(x_0)$ if and only if $(p,X)\in
J^{2,+}\bar{u}(0)$.
\end{remark}
We also use the following notations.
\begin{itemize}
       \item $\USC (M)=\{\textrm {upper semicontinuous functions on
       $M$}\},$
       \item $\LSC (M)=\{\textrm {lower semicontinuous functions on
       $M$}\}.$
     \end{itemize}
\begin{definition}A viscosity subsolution of $F=0$ on $M$ is a
function $u\in \USC(M)$ such that
\begin{equation}\label{eq04}
    F(x,u,p,X)\le 0 \textrm{ for all }x\in M \textrm{ and }(p,X)\in
    J^{2,+}u(x).
\end{equation}
A viscosity supersolution of $F=0$ on $M$ is a function $u\in
\LSC(M)$ such that
\begin{equation}\label{eq05}
    F(x,u,p,X)\ge 0 \textrm{ for all }x\in M \textrm{ and }(p,X)\in
    J^{2,-}u(x).
\end{equation}
$u$ is a viscosity solution of $F=0$ on $M$ if it is both a
viscosity subsolution and a viscosity supersolution of $F=0$ on $M$.

\end{definition}
Following \cite{cil}, we can similarly define $\bar J^{2,+}u(x)$,
$\bar J^{2,-}u(x)$ as
\begin{equation*}
    \bar J^{2,+}u(x)=\left\{\begin{split}
                                 & (p,X)\in T_{x_0}^{*}M\times\mathcal{S}^2 T_{x_0}^{*}M, \textrm{ such that } (x_0,u(x_0), p, X)
    \textrm{ is a limit point of } \\ &(x_k,u(x_k),p_k, X_k)\in J^{2,+}u(x_k)
    \textrm{ in the  topology
    of  }\mathcal{F}(M),
    \end{split}\right\}
\end{equation*}
and
\begin{equation*}
    \bar J^{2,-}u(x)=\left\{\begin{split}
                                 & (p,X)\in T_{x_0}^{*}M\times\mathcal{S}^2 T_{x_0}^{*}M, \textrm{ such that } (x_0,u(x_0), p, X)
    \textrm{ is a limit point of } \\ &(x_k,u(x_k),p_k, X_k)\in J^{2,-}u(x_k)
    \textrm{ in the  topology
    of  }\mathcal{F}(M).
    \end{split}\right\}
\end{equation*}

\section{Hessian-type comparison Theorem for
$\dist(x,y)^2$}\label{secComp}

One of the key points  to generalize  the  maximum principle to a
complete Riemannian manifold  is a  new Hessian comparison theorem
for Riemannian manifold. This is essentially an application of
second variational formula for the arclength.

Let $M$ be a Riemannian manifold. We define the function of square
of the distance function as $\varphi: M\times M\to \mathbb{R}$ as
\begin{equation*}
    \varphi(x,y)=\dist (x,y)^2.
\end{equation*}
It is well known that this function is smooth when $\dist (x,y)$ is
small. We will prove
\begin{theorem}\label{thmhess}Let $M$ be a connected Riemannian manifold with sectional curvature bounded below by $-\kappa^2$.
 Given two points $x,
y\in M$ with $\dist(x,y)< \min\{i(x), i(y)\}$, $\gamma:[0,l]\to M$
is the unique geodesic of unit speed with $\gamma(0)=x$ and
$\gamma(l)=y$. Denote by $P_{\gamma}(t):T_{\gamma(0)}M\to
T_{\gamma(t)}M$ the parallel transport  along $\gamma$. Then any two
vectors $V_1$ and $V_2$ satisfying $\langle V_1,
\gamma'(0)\rangle=\langle V_2, \gamma'(l)\rangle=0$,  the Hessian of
the square  of distance function $\varphi$ on $M\times
 M$ satisfies

\begin{equation}\label{eq7}
\begin{split}
    D^2\varphi((V_1,V_2),(V_1,V_2)) \le &2l\kappa[\coth \kappa l\langle V_1,V_1\rangle+\coth \kappa l\langle
    V_2,V_2\rangle]\\
    &-2l\kappa[ \frac{2}{\sinh \kappa
    l}\langle V_2, P_{\gamma}(l)V_1\rangle],
    \end{split}
\end{equation}
Particularly,
\begin{equation}\label{eq8}
\begin{split}
    D^2\varphi((V_1,P_{\gamma}(l)V_1),(V_1,P_{\gamma}(l)V_1)) \le &4|V_1|^2\kappa l\tanh\frac{\kappa l}2
    .
    \end{split}
\end{equation}

\end{theorem}
Before proving the theorem we give some remarks.
\begin{remark} When $M$ is the  hyperbolic space $\mathbb{H}^n(-\kappa^2)$, we can also write the $D^2\varphi$ on the subspace
$\{\gamma(0)\}^{\bot}\times\{\gamma(l)\}^{\bot}\subset T_xM\times
T_yM$ as
\begin{equation*}
    2l\kappa\left(
       \begin{array}{cc}
         I\coth \kappa l & -\frac{1}{\sinh \kappa l} I\\
         -\frac{1}{\sinh \kappa l}I & I\coth \kappa l \\
       \end{array}
     \right).
\end{equation*}

\end{remark}
\begin{remark} The estimates in theorem is sharp in sense that all inequalities becomes equalities for space forms. The theorem improves significantly  Proposition 3.3
in \cite{afs} also part (1) of Proposition 3.1.
\end{remark}

\begin{remark} Note that we allow  $\kappa$ to be an imaginary
 and in case that the curvature is bounded below by a positive
constant, we can get a corresponding estimate.
\end{remark}
\begin{remark} We  also prove a version similar to Laplacian comparison theorem when we have a Ricci curvature lower
bound. We will see that it is useful in applications.
\end{remark}
\begin{theorem}Let $M$ be a connected Riemannian manifold. Given two points $x,
y\in M$ with $\dist(x,y)< \min\{i(x), i(y)\}$, $\gamma:[0,l]\to M$
is the unique geodesic of unit speed with $\gamma(0)=x$ and
$\gamma(l)=y$, then for any two vectors $V_1\in T_xM, V_2\in T_yM$
$D^2\varphi((V_1,V_2),(V_1,V_2))$
 satisfying $\langle V_1,
\gamma'(0)\rangle=\langle V_2, \gamma'(l)\rangle=0$
\begin{equation}\label{eq0}
    D^2\varphi((V_1,V_2),(V_1,V_2)) =2l\int_0^l[|\nabla_{\gamma'}J|^2-\langle
R(\gamma',J)\gamma',J\rangle] dt,
\end{equation}
where  $J$ is the Jacobi field along $\gamma$ with $J(0)=V_1$ and
$J(l)=V_2.$

\end{theorem}

Recall that given a $C^2$ curve $\gamma:[a,b]\to M$. A variation
$\eta$ of $\gamma$ is a $C^2$ mapping $\eta:[a,b]\times
(-\epsilon_0, \epsilon_0)\to M$ for some $\epsilon_0>0,$  for which
$\gamma(t)=\eta(t,0)$ for all $t\in [a,b]$. We write
$\partial_t\eta,
\partial_{\epsilon}\eta $ for $\eta_{*}(\partial_t),
\eta_{*}(\partial_{\epsilon})$ respectively, and denote
differentiation of vector field along $\eta$ with respect to
$\partial_t$, $\partial_{\epsilon}$ by $\nabla_t, \nabla_{\epsilon}$
respectively. Then the length of $\eta_{\epsilon}$, namely
\begin{equation*}
    L(\epsilon)=\int_a^b|\partial_t\eta(t,\epsilon)|dt,
\end{equation*}
is differentiable and
\begin{equation}\label{eq1}
\frac{dL}{d\epsilon}=\langle\partial_{\epsilon}\eta,\partial_t\eta/|\partial_t\eta|\rangle|_a^b-
\int_a^b\langle\partial_{\epsilon}\eta,\nabla_t(\partial_t\eta/|\partial_t\eta|)\rangle
dt.
\end{equation}
In particular, if $\eta$ is parameterized with respect to arc
length, and we set $V(t)=(\partial_\epsilon\eta)(t,0)$, then for the
first derivative of $L$ we have
\begin{equation}\label{eq2}
\frac{dL}{d\epsilon}(0)=\langle\partial_{\epsilon}\eta,\gamma'\rangle
|_a^b- \int_a^b\langle\partial_{\epsilon}\eta,\nabla_t\gamma'\rangle
dt,
\end{equation}
and for the second derivative of $L$ we have
\begin{equation}
\begin{split}\label{eq3}
\frac{d^2L}{d\epsilon^2}(0)=&\langle\nabla_{\epsilon}\partial_{\epsilon}\eta|_{\epsilon=0},\gamma'\rangle |_a^b-\\
&\int_a^b[|\nabla_tV|^2+\langle
R(\gamma',V)\gamma',V\rangle-\langle\nabla_{\epsilon}\partial_{\epsilon}\eta,\nabla_t\gamma'\rangle-\langle\gamma',\nabla_tV\rangle^2]
dt.
\end{split}
\end{equation}
Let $x,y$ be two non conjugate points connected by a unique
minimizing geodesic $\gamma: [0,l]\to M$ with $l=\dist (x,y)\le
\min\{i(x), i(y)\}$. Given two vectors $V_1\in T_xM, V_2\in T_yM$,
there exist a unique Jacobi field $J:[0,l]\to TM$ such that
$J(0)=V_1$ and $J(l)=V_2$. Thus we can construct a variation  $\eta$
of $\gamma$,
 $\eta:[0,l]\times (-\epsilon_0, \epsilon_0)\to M$ for some
$\epsilon_0>0,$  for which $\gamma(t)=\eta(t,0)$ for all $t\in
[0,l]$ and  $\partial_{\epsilon}\eta(0,0)=V_1$ and
$\partial_{\epsilon}\eta(l,0)=V_2$. We know from the properties of
Jacobi field that $\eta(\cdot, \epsilon)$ is a geodesic for each
$\epsilon\in (-\epsilon_0, \epsilon_0)$. So
\begin{equation}\label{eq4}
 L(\epsilon)=\int_a^b|\partial_t\eta(t,\epsilon)|dt=\dist
 (\eta(0,\epsilon), \eta(l,\epsilon)).
\end{equation}
Thus, by the first variational formula, we can calculate the
differential of $\varphi$,
\begin{equation}\label{eq5}
    D\varphi(V_1,V_2)=2l[\langle V_2,\gamma'(l)\rangle-\langle
    V_1,\gamma'(0)\rangle].
\end{equation}
To compute the second differential, we only need to compute
$D^2\varphi((V_1,V_2),(V_1,V_2))$ for $V_1$ and $V_2$ satisfying
$\langle V_1, \gamma'(0)\rangle=\langle V_2, \gamma'(l)\rangle=0$.
In this case, $J$ are normal Jacobi fields.  Note that from the
construction of variation from the Jacobi field, both
$\eta(0,\cdot)$ and $\eta(l,\cdot)$ are geodesics and we are working
in a convex neighborhood. So
\begin{equation}\label{eq6}\begin{split}
                             D^2\varphi((V_1,V_2),(V_1,V_2)) &=2l\frac{d^2L}{d\epsilon^2}(0) \\
                               & =2l\int_0^l[|\nabla_tJ|^2-\langle
R(\gamma',J)\gamma',J\rangle] dt.
                           \end{split}
\end{equation}

\begin{proof}Since the Jacobi fields minimize the index form $I(X,X)$ with the
same boundary conditions $X(0)=V_1$ and $X(l)=V_2$, where
\begin{equation}\label{eq7.1}
    I(X,X)=\int_0^l[|\nabla_{\gamma'}X|^2-\langle
R(\gamma',X)\gamma',X\rangle] dt.
\end{equation}
We can obtain  vector fields $V_1(t)$ and $V_2(t)$ by parallel
transport of $V_1$ and $V_2$ respectively. Choose
\begin{equation*}
    X(t)=(\cosh\kappa t-\coth\kappa l\sinh\kappa t)V_1(t)+\frac{\sinh\kappa
    t}{\sinh\kappa l}V_2(t).
\end{equation*}
Then
\begin{equation}\label{eqn8}
\begin{split}
  I(X,X)&=\langle\nabla_{\gamma'}X,X\rangle|_0^l-\int_0^l[\langle\nabla_{\gamma'}\nabla_{\gamma'}X,X\rangle+\langle
R(\gamma',X)\gamma',X\rangle] dt  \\
    & \le \langle\nabla_{\gamma'}X,X\rangle|_0^l\\
    &=\langle\kappa(\sinh\kappa l-\coth \kappa l\cosh\kappa
    l)V_1(l),V_2\rangle+\\
    &\qquad\langle\kappa\coth \kappa l V_2,
    V_2\rangle-\langle-\kappa\coth \kappa l V_1+\frac{\kappa}{\sinh \kappa
    l}V_2(0), V_1\rangle\\
    &=\kappa\coth \kappa l\langle V_1,V_1\rangle+\kappa\coth \kappa l\langle
    V_2,V_2\rangle- \frac{2\kappa}{\sinh \kappa
    l}\langle V_2(0), V_1\rangle.
\end{split}
\end{equation}
The theorem follows from
\begin{equation*}
    D^2\varphi((V_1,V_2),(V_1,V_2))\le 2lI(X,X).
\end{equation*}
when $V_2=P_{\gamma}(l)V_1$, we have

\begin{equation}\label{eq8.1}
\begin{split}
    D^2\varphi((V_1,P_{\gamma}(l)V_1),(V_1,P_{\gamma}(l)V_1)) \le &2l\kappa[\coth \kappa l\langle V_1,V_1\rangle+\coth \kappa l\langle
    P_{\gamma}(l)V_1),P_{\gamma}(l)V_1)\rangle]\\
    &-2l\kappa[ \frac{2}{\sinh \kappa
    l}\langle P_{\gamma}(l)V_1), P_{\gamma}(l)V_1\rangle]\\
    =&4l\kappa\tanh\frac{\kappa l}2 |V_1|^2.
    \end{split}
\end{equation}

\end{proof}

\begin{theorem}\label{thmlaplace}Let $M$ be a  Riemannian manifold with Ricci curvature bounded below by $-(n-1)\kappa^2$.
 Given two points $x,
y\in M$ with $\dist(x,y)< \min\{i(x), i(y)\}$, $\gamma:[0,l]\to M$
is the unique geodesic of unit speed with $\gamma(0)=x$ and
$\gamma(l)=y$. For any normal base at $x$, $\{\gamma'(0),
e_2,\cdots,e_n \}$.
\begin{equation}\label{eq8.2}
\begin{split}
    \sum_{i=2}^{n-1}D^2\varphi((e_i,P_{\gamma}(l)e_i),(e_i,P_{\gamma}(l)e_i) \le &4(n-1)\kappa l\tanh\frac{\kappa l}2
    .
    \end{split}
\end{equation}

\end{theorem}

\begin{proof}Similar to the above proof, we use the property that  the Jacobi field minimize the index form $I(X,X)$ with the
same boundary conditions $X(0)=V_1$ and $X(l)=V_2$. We can obtain
vector fields $e_i(t)$  by parallel transport of $e_i$. For any
 function $\psi\in C^2[0,l]$ with $\psi(0)=\psi(l)=1$, we
choose
\begin{equation*}
    X_i(t)=\psi(t)e_i(t).
\end{equation*}
Then
\begin{equation*}
\begin{split}
    D^2\varphi((e_i,P_{\gamma}(l)e_i),(e_i,P_{\gamma}(l)e_i))&\le 2l\int_0^l[|\nabla_{\gamma'}X_i|^2-\langle
R(\gamma',X_i)\gamma',X_i\rangle] dt  \\
&=2l\int_0^l[|\psi'|^2-\psi^2K(\gamma',e_i)] dt,
\end{split}\end{equation*}
 where $K(X,Y)$ is the sectional curvature of the plane spanned by vectors $X,Y$. Then  we have
\begin{equation*}
\begin{split}
    \sum_{i=2}^nD^2\varphi((e_i,P_{\gamma}(l)e_i),(e_i,P_{\gamma}(l)e_i))
&\le2l\int_0^l[(n-1)|\psi'|^2-\psi^2\Ric(\gamma',\gamma')] dt\\
&\le 2(n-1)l\int_0^l[|\psi'|^2+\kappa^2\psi^2] dt.
\end{split}\end{equation*}
where $\Ric$ denotes the Ricci curvature. Let
\begin{equation*}
    \psi(t)=\cosh \kappa t+\frac{1-\cosh\kappa l}{\sinh\kappa
    l}\sinh \kappa t.
\end{equation*}
The claimed result follows from  a direct computation.

\end{proof}

\section{The maximum principle for viscosity solutions}
\subsection{Dirichlet problem case}
The same proof of Lemma 3.1 in \cite{cil} gives the following lemma.
\begin{lemma}Let $M$ be a compact Riemannian manifold with or
without boundary, $u\in \USC(M)$ $v\in \LSC(M)$ and
\begin{equation}\label{eq9}
    \mu_{\alpha}=\sup\{u(x)-v(y)-\frac{\alpha}2\dist(x,y)^2, x,y\in M \}
\end{equation}
for $\alpha>0.$ Assume  that $\mu_{\alpha}<+\infty$ for large
$\alpha$ and $(x_{\alpha},y_{\alpha})$  satisfy
\begin{equation}\label{eq10}
    \lim_{\alpha\to\infty}[\mu_{\alpha}-(u(x_{\alpha})-v(y_{\alpha})-\frac{\alpha}2\dist(x_{\alpha},y_{\alpha})^2)]=0.
\end{equation}
Then the following holds:
\begin{equation}\label{eq11}
    \left\{
      \begin{array}{ll}
        (i)\quad\lim_{\alpha\to\infty}\alpha\dist(x_{\alpha},y_{\alpha})^2=0,\quad and & \hbox{} \\
        (ii)\quad \lim_{\alpha\to\infty}\mu_{\alpha}=u(x_0)-v(x_0)=\sup(u(x)-v(x)), & \hbox{} \\
         \textrm{ where } x_0=\lim_{\alpha\to\infty}x_{\alpha}. & \hbox{}
      \end{array}
    \right.
\end{equation}

\end{lemma}
From Remark \ref{rk1} we can deduce from Theorem 3.2 in \cite{cil}
the following lemma.
\begin{lemma} Let $M_1$, $M_2$ be  Riemannian manifolds with
or without boundary, $u_1\in \USC(M_1)$, $u_2\in \LSC(M_2)$ and
$\varphi\in C^2(M_1\times M_2)$. Suppose $(\hat{x},\hat{y})\in
M_1\times M_2$ is a local maximum of $u_1(x)-u_2(y)-\varphi(x,y)$.
Then for any $\varepsilon>0$ there exist $X_1\in \mathcal{S}
T_{\hat{x}}^*M_1$ and $X_2\in \mathcal{S}^2T_{\hat{y}}^*M_2$  such
that
\begin{equation*}
    (D_{x_i}\varphi(\hat{x},\hat{y}), X_i)\in \bar
    J^{2,+}u_i(\hat{x}_i), \textrm{ for } i=1,2,
\end{equation*}
and the block diagonal matrix satisfies
\begin{equation}\label{eq12}
    -\left(\frac1{\varepsilon}+\|A\|\right)I\le\left(
                                               \begin{array}{cc}
                                                 X_1 & 0 \\
                                                 0 & -X_2 \\
                                               \end{array}
                                             \right)
    \le A+\varepsilon A^2
\end{equation}
where $A=D^2\varphi(\hat{x},\hat{y})\in\mathcal{S}^2T^*(M_1\times
M_2)$ and
\begin{equation}\label{eq12.5}
\|A\|=\sup\{| A(\xi,\xi)|, \xi\in T_{(\hat{x},\hat{y})}M_1\times
M_2, |\xi|=1\}.
\end{equation}

\end{lemma}
Now we are in a position to prove our comparison theorem.
\begin{theorem}\label{thmDiri}
Let $M$  be a  compact Riemannian manifold with or without  boundary
$\partial M$, $\mathcal{F}(M)\in C(F,\rr)$  a proper function
satisfying
\begin{equation}\label{eq12.6}
    \beta(r-s)\le F(x,r,p,X)-F(x,s,p,X) \textrm{ for } r\ge s,
\end{equation}
for some positive constant $\beta$ and the following condition (H):
\begin{quote}
there exists a function  $\omega:[0,\infty]\to [0,\infty]$ satisfying $\omega(t)>0$ when $ t>0,$ and
$\omega(0+)=0$ such that
\begin{equation}\label{eq12.7}
F(y,r, -\alpha l\iota(\gamma'(l)),X_{2})-F(x,r, -\alpha
l\iota(\gamma'(0)),X_{1})\le \omega(\alpha\dist(x,y)^2+\dist(x,y))
\end{equation}
 for $ X_1\in \mathcal{S}
T_{{x}}^*M$ and $X_2\in \mathcal{S}^2T_{{y}}^*M$,  satisfying
$X_1\le X_2\circ P_{\gamma}(l)$. Here $\iota$ is the dual map
between the tangent and cotangent bundles.
\end{quote}

 Let $u_1\in \USC(\bar
{M})$ and $u_2\in \LSC(\bar {M})$ be a subsolution and supersolution
of $F=0$ respectively.  Then $u_1-u_2$ cannot achieve a positive
local maximum at any interior point. In particular if $M$  is a
compact Riemannian manifold with boundary $\partial M$ and if
$u_1\le u_2$ on $\partial M$, then $u_1\le u_2$ on $\bar{ M}$.

\end{theorem}
\begin{proof}For the sake of a contradiction we assume that
$2\delta:=\sup_{x\in M}\{u_1(x)-u_2(x)\}>0$. Then for $\alpha$
large,
\begin{equation*}%\label{eq9}
    0<2\delta\le\mu_{\alpha}:=\sup\{u_1(x)-u_2(y)-\frac{\alpha}2\dist(x,y)^2), x,y\in M
    \}<+\infty.
\end{equation*}
 There exists $(x_{\alpha},y_{\alpha})$ such that
\begin{equation*}%\label{eq10}
    \mu_{\alpha}-(u_1(x_{\alpha})-u_2(y_{\alpha})-\frac{\alpha}2\dist(x_{\alpha},y_{\alpha})^2)=0.
\end{equation*}
Then the following holds:
\begin{equation*}%\label{eq11}
    \left\{
      \begin{array}{ll}
        (i)\quad\lim_{\alpha\to\infty}\alpha\dist(x_{\alpha},y_{\alpha})^2=0,\quad and & \hbox{} \\
        (ii)\quad \lim_{\alpha\to\infty}\mu_{\alpha}=u_1(x_0)-u_2(x_0)=2\delta, & \hbox{} \\
         \textrm{ where } x_0=\lim_{\alpha\to\infty}x_{\alpha}. & \hbox{}
      \end{array}
    \right.
\end{equation*}
We can suppose that there exists a convex neighborhood $D$ of $x_0$
such that $x_{\alpha}\in D$ and $y_{\alpha}\in D$ for all $\alpha$
large.  So there exists a constant $\kappa$ such that the sectional
curvature of $M$ is bounded below by a constant $-\kappa^2$. Without
loss of generality, we assume the diameter $L(D)$ of D is small such
that $\cosh \kappa L(D)\le 2$. Let
$\gamma_{\alpha}:[0,l_{\alpha}]\to D$ be the unique normal geodesic
joining $x_{\alpha}$ and $y_{\alpha}$, where $l_{\alpha}=\dist
(x_{\alpha}, y_{\alpha})$. In this case
\begin{equation*}
    \nabla\varphi=(-2l\gamma'(0),2l\gamma'(l_{\alpha})).
\end{equation*}
We can decompose $T_xM\times T_yM$  as
$(\gamma'(0)\rr\oplus\{\gamma'(0)\}^{\bot})\times(\gamma'(0)\rr\oplus\{\gamma'(l_{\alpha})\}^{\bot})$.
According this decomposition and from Theorem \ref{thmhess}, we know
$A_{\alpha}:=D^2(\frac{\alpha}2\dist^2)$ satisfies
\begin{equation*}
    A_{\alpha}\le \alpha \left(
       \begin{array}{cccc}
         1&&-1&\\
         &Il_{\alpha}\kappa\coth \kappa l_{\alpha} & &-\frac{l_{\alpha}\kappa}{\sinh \kappa l_{\alpha}} I\\
         -1&&1&\\
         &-\frac{l_{\alpha}\kappa}{\sinh \kappa l_{\alpha}}I && Il_{\alpha}\kappa\coth \kappa l_{\alpha} \\
       \end{array}
     \right),
\end{equation*}
where $I$ is $(n-1)\times (n-1)$ unit matrix.
 Then
 \begin{equation*}
 \|A_{\alpha}\|\le \alpha\frac{\kappa l_{\alpha}}{\sinh \kappa l_{\alpha}}(\cosh kl_{\alpha}+1)\le 3\alpha,
 \end{equation*}
Here we have assumed that $\kappa l_\alpha$ is not big.
  For any $\varepsilon>0$ there
exist  $X_{1\alpha}\in \mathcal{S}^2T_{x_{\alpha}}^*M$ and
$X_{2\alpha}\in \mathcal{S}^2T_{y_{\alpha}}^*M$ such that
 the block diagonal matrix satisfies
\begin{equation}\label{eq13}
    -\left(\frac1{\varepsilon}+\|A_{\alpha}\|\right)I\le\left(
                                               \begin{array}{cc}
                                                 X_{1\alpha} & 0 \\
                                                 0 & -X_{2\alpha} \\
                                               \end{array}
                                             \right)
    \le A_{\alpha}+\varepsilon A^2_{\alpha}.
\end{equation}
Choosing $\varepsilon=\frac{1}{\alpha(\cosh \kappa L(D)+1)}$ we have
the following
\begin{equation}\label{eq14}
    -4\alpha I\le\left(
                                               \begin{array}{cc}
                                                 X_{1\alpha} & 0 \\
                                                 0 & -X_{2\alpha} \\
                                               \end{array}
                                             \right)
    \le 4\alpha \left(
       \begin{array}{cccc}
         1&&-1&\\
         &Il_{\alpha}\kappa\coth \kappa l_{\alpha} & &-\frac{l_{\alpha}\kappa}{\sinh \kappa l_{\alpha}} I\\
         -1&&1&\\
         &-\frac{l_{\alpha}\kappa}{\sinh \kappa l_{\alpha}}I && Il_{\alpha}\kappa\coth \kappa l_{\alpha} \\
       \end{array}
     \right),
\end{equation}
Let us denote $\iota$ the dual map between the tangent and cotangent
bundles. We have proved that for a sufficiently large $\alpha$ at a
maximum $(x_{\alpha},y_{\alpha})$ of
$u_1(x)-u_2(y)-\frac{\alpha}2\dist(x_{\alpha},y_{\alpha})^2$ there
exist  $(-\alpha l_\alpha\iota(\gamma_{\alpha}'(0)), X_{1\alpha})\in
\bar{J}_M^{2,+}(u_1(x_{\alpha})$ and $(-\alpha
l_\alpha\iota(\gamma_{\alpha}'(l_{\alpha})), X_{2\alpha})\in
\bar{J}_M^{2,-}(u_2(y_{\alpha})$ such that $X_{2\alpha}\ge
X_{1\alpha}\circ P_{\gamma}(l)- 4\alpha
                      l_{\alpha}^2\kappa^2 Id$ holds.

Since $u_1$ and $u_2$ are subsoltion and supersoltion  we have
\begin{equation}\label{eq15}
    F(x_{\alpha},u_1(x_{\alpha}),
    -\alpha l_\alpha\iota(\gamma'_{\alpha}(0)),X_{1\alpha})\le 0\le F(y_{\alpha},u_2(y_{\alpha}),
    -\alpha l_\alpha\iota(\gamma'_{\alpha}(l_{\alpha})),X_{2\alpha}).
\end{equation}
Then for $\alpha$ sufficiently large
\begin{equation*}
    \begin{split}
                    0<\beta\delta & <\beta(u_1(x_{\alpha})-u_2(y_{\alpha})-\frac{\alpha}2\dist^2(x_{\alpha},y_{\alpha})) \\
                      &\le F(x_{\alpha},u_1(x_{\alpha}),-\alpha l_\alpha\iota(\gamma'_{\alpha}(0)),X_{1\alpha})-F(x_{\alpha},u_2(y_{\alpha}),-\alpha l_\alpha\iota(\gamma'_{\alpha}(0)),X_{1\alpha})-\frac{\beta\alpha}2\dist^2(x_{\alpha},y_{\alpha})\\
                      &=F(x_{\alpha},u_1(x_{\alpha}),-\alpha l_\alpha\iota(\gamma'_{\alpha}(0)),X_{1\alpha})-F(y_{\alpha},u_2(y_{\alpha}),-\alpha l_\alpha\iota(\gamma'_{\alpha}(l_{\alpha})),X_{1\alpha}\circ P_{\gamma}(0))+\\
                      &\qquad F(y_{\alpha},u_2(y_{\alpha}),-\alpha l_\alpha\iota(\gamma'_{\alpha}(l_{\alpha})),X_{2\alpha})-F(y_{\alpha},u_2(y_{\alpha}),-\alpha l_\alpha\iota(\gamma'_{\alpha}(l_{\alpha})),X_{2\alpha})+\\
                      &\qquad F(y_{\alpha},u_2(y_{\alpha}),-\alpha l_\alpha\iota(\gamma'_{\alpha}(l_{\alpha})),X_{1\alpha}\circ P_{\gamma}(0))-F(x_{\alpha},u_2(y_{\alpha}),-\alpha l_\alpha\iota(\gamma'_{\alpha}(0)),X_{1\alpha})-\frac{\beta\alpha}2\dist^2(x_{\alpha},y_{\alpha})\\
                      &\le\omega(\alpha\dist(x_{\alpha},y_{\alpha})^2+\dist(x_{\alpha},y_{\alpha}))-\frac{\beta\alpha}2\dist^2(x_{\alpha},y_{\alpha})+\\
                      &\qquad
                      -F(y_{\alpha},u_2(y_{\alpha}),-\alpha l_\alpha\iota(\gamma'_{\alpha}(l_{\alpha})),X_{1\alpha}\circ P_{\gamma}(0))+F(y_{\alpha},u_2(y_{\alpha}),-\alpha l_\alpha\iota(\gamma'_{\alpha}(l_{\alpha})),X_{2\alpha})\\
                      &\le\omega(\alpha\dist(x_{\alpha},y_{\alpha})^2+\dist(x_{\alpha},y_{\alpha}))-\frac{\beta\alpha}2\dist^2(x_{\alpha},y_{\alpha})+\\
                      &\qquad
                      -F(y_{\alpha},u_2(y_{\alpha}),-\alpha l_\alpha\iota(\gamma'_{\alpha}(l_{\alpha})),X_{2\alpha}+4\alpha
                      l_{\alpha}^2\kappa^2I)+F(y_{\alpha},u_2(y_{\alpha}),-\alpha l_\alpha\iota(\gamma'_{\alpha}(l_{\alpha})),X_{2\alpha}).
                 \end{split}
\end{equation*}
Let $\alpha\to +\infty$. Since $F$ is continuous in $X$, the left
hand side goes to zero and we have arrived at a contradiction.
\end{proof}

\section{Viscosity maximum principle on complete  Riemannian manifolds}
Maximum principles for $C^2$ functions on complete Riemannian
manifolds were already done by Omori \cite{O} and Yau \cite{Y}, but
our approach and the spirit of the results are quite different. In
\cite{Y}, Yau developed gradient estimate for $C^2$ functions
satisfing some inequalities and proved a maximum principle which is
well known and has many important applications in geometry. In this
section we will generalize both Omori and Yau's maximum principle to
non-differentiable functions. The approach is quite different. Even
for  $C^2$ function we present here a new proof for Omori-Yau's
maximum principle.
\begin{theorem}\label{thmOmori}
Let $M$ be a complete Riemannian manifold with sectional curvature
bounded below by a constant $-\kappa^2$. Let $u
\in{\mathrm{USC}}(M)$, and  $v\in{\mathrm{LSC}}(M)$ be two functions
satisfying
 \begin{equation}\label{}
    \mu_0:=\sup_{x\in M}[u(x)-v(x)]< +\infty.
 \end{equation}
 Assume that $u$ and $v$ are bounded from above and below respectively and there exists a  function  $\omega:[0,\infty]\to [0,\infty]$ satisfying $\omega(t)>0$ when $ t>0,$ and
$\omega(0+)=0$  such that
\begin{equation}\label{}
    u(x)-u(y)\le \omega(\dist(x,y)).
\end{equation}
Then for each $\varepsilon>0$, there exist $x_{\varepsilon},
y_{\varepsilon}\in M$, such that
$(p_{\varepsilon},X_{\varepsilon})\in \bar{J}%
^{2,+}u(x_{\varepsilon}),\  \ (q_{\varepsilon},Y_{\varepsilon})\in
\bar{J} ^{2,-}v(y_{\varepsilon}),$ such that
\[
u(x_{\varepsilon})-v(y_{\varepsilon})\geq \mu_{0}-\varepsilon,\
\]
and such that
\[
\dist(x_{\varepsilon},y_{\varepsilon})<\varepsilon,\  \  \
|p_{\varepsilon }-q_{\varepsilon}\circ P_{\gamma}(l)|<\varepsilon,\
\ X_{\varepsilon}\leq Y_{\varepsilon }\circ
P_{\gamma}(l)+\varepsilon  P_{\gamma}(l),\
\]
where $l=\dist(x_{\varepsilon},y_{\varepsilon})$ and $
P_{\gamma}(l)$ is the parallel transport along the shortest geodesic
connecting $x_{\varepsilon}$ and $y_{\varepsilon}$.

\end{theorem}
\begin{proof} We divide the proof into two parts.

{\bf Part 1.} Without loss of generality, we assume that $\mu_0>0$.
Otherwise we replace $u$ by $u-\mu_0+1$. For each $\alpha>0$, we
take $\hat{x}_{\alpha}\in M$ such that
\begin{equation*}
    u(\hat{x}_{\alpha})-v(\hat{x}_{\alpha})+\omega(\sqrt{\frac{\mu_0}{\alpha}})\geq
    \mu_0.
\end{equation*}
Let $\lambda_{\alpha}=-\frac{1}{\ln
\omega(\sqrt{\frac{\mu_0}{\alpha}})}$. We consider the following
maximization
\begin{equation}\label{}
    \sigma_{\alpha}=\sup_{x,y\in
    M}[u(x)-v(y)-\frac{\alpha}2\dist(x,y)^2-\frac{\lambda_{\alpha}}2\dist(\hat{x}_{\alpha},x)^2-\frac{\lambda_{\alpha}}2\dist(\hat{x}_{\alpha},y)^2]
\end{equation}
Taking $\alpha$ large if necessary, we have $(x_{\alpha},
y_{\alpha})$ satisfying
\begin{equation*}
    \sigma_{\alpha}=u(x_{\alpha})-v(y_{\alpha})-\frac{\alpha}2\dist(x_{\alpha},y_{\alpha})^2-\frac{\lambda_{\alpha}}2\dist(\hat{x}_{\alpha},x_{\alpha})^2-\frac{\lambda_{\alpha}}2\dist(\hat{x}_{\alpha},y_{\alpha})^2
\end{equation*}
Let
\begin{equation}\label{}
    \sigma^o_{\alpha}=\sup_{x,y\in
    M}[u(x)-v(x)-\lambda_{\alpha}\dist(\hat{x}_{\alpha},x)^2]
\end{equation}
It is straightforward to see that
\begin{equation}\label{}
    \sigma_{\alpha}\ge\sigma^o_{\alpha}\ge u(\hat{x}_{\alpha})-v(\hat{x}_{\alpha})\geq
    \mu_0-\omega(\sqrt{\frac{\mu_0}{\alpha}})
\end{equation}
and it follows from the boundedness of $u$ and $v$  that
$\sigma_{\alpha}$ is bounded  from below by some constant
$\sigma_{*}\leq \sigma^o_{\alpha}$. Then there exists a constant $C$
such that
\begin{equation}\label{eqt5}
\frac{\alpha}2\dist(x_{\alpha},y_{\alpha})^2\le C\end{equation} and
\begin{equation}\label{eqt6}
\begin{split}
                          \frac{\alpha}2\dist(x_{\alpha},y_{\alpha})^2+&\frac{\lambda_{\alpha}}2\dist(\hat{x}_{\alpha},x_{\alpha})^2+\frac{\lambda_{\alpha}}2\dist(\hat{x}_{\alpha},y_{\alpha})^2=  u(x_{\alpha})-v(y_{\alpha})-\sigma_{\alpha} \\
                            & \le
                            u(x_{\alpha})-u(y_{\alpha})+u(y_{\alpha})-v(y_{\alpha})-\mu_0+\omega(\sqrt{\frac{\mu_0}{\alpha}})\\
                            & \le
                            \omega(\dist (x_{\alpha},y_{\alpha}) )+u(y_{\alpha})-v(y_{\alpha})-\mu_0+\omega(\sqrt{\frac{\mu_0}{\alpha}})\\
                            & \le
                            \omega(\dist (x_{\alpha},y_{\alpha}))+\omega(\sqrt{\frac{\mu_0}{\alpha}})\rightarrow
                            0.
                        \end{split}
\end{equation}
So (\ref{eqt5}) can be improved as
$\frac{\alpha}2\dist(x_{\alpha},y_{\alpha})^2\le \frac{\mu_0}2$.
Using the same process of (\ref{eqt6}), we have
\begin{equation*}
    \frac{\alpha}2\dist(x_{\alpha},y_{\alpha})^2+\frac{\lambda_{\alpha}}2\dist(\hat{x}_{\alpha},x_{\alpha})^2+\frac{\lambda_{\alpha}}2\dist(\hat{x}_{\alpha},y_{\alpha})^2\le
    2\omega(\sqrt{\frac{\mu_0}{\alpha}}).
\end{equation*}

Thus, when $\alpha \rightarrow \infty$,
\begin{equation}\label{}
\begin{split}
&\dist(x_{\alpha},y_{\alpha})^{2}\leq
\frac{4\omega(\sqrt{\frac{\mu_{0} }{\alpha}})}{\alpha}\rightarrow0,\\
&\dist(x_{\alpha},\hat{x}_{\alpha})^{2}+\dist(y_{\alpha
},\hat{x}_{\alpha})^{2}\leq
\frac{4\omega(\sqrt{\frac{\mu_{0}}{\alpha}} )}{ \lambda_{\alpha}}
\rightarrow0.
  \end{split}
\end{equation}

{\bf Part 2.} We apply now Theorem 3.2 in \cite{cil} to
$\varphi_{\alpha}(x,y)=\frac{\alpha}2\dist(x,y)^2+\frac{\lambda_{\alpha}}2\dist(\hat{x}_{\alpha},x)^2+\frac{\lambda_{\alpha}}2\dist(\hat{x}_{\alpha},y)^2$.
We have for any $\delta>0$ there exist $X_{\alpha}\in \mathcal{S}
T_{{x_{\alpha}}}^{\ast}M$ and $Y_{\alpha}\in \mathcal{S}
T_{y_{\alpha}}^{\ast}M$ such that
\[
(D_{x}\varphi_{\alpha}({x_{\alpha}},{y_{\alpha}}),X_{\alpha})\in
\bar{J}^{2,+}u(x_{\alpha}),\text{ and }
(-D_{y}\varphi_{\alpha}({x_{\alpha}},{y_{\alpha}}),Y_{\alpha})\in
\bar{J}^{2,-}v(y_{\alpha}),
\]
and the block diagonal matrix satisfies
\begin{equation}
-\left(  \frac{1}{\delta}+\Vert A_{\alpha}\Vert \right)  I\leq
\left(
\begin{array}
[c]{cc}%
X_{\alpha} & 0\\
0 & -Y_{\alpha}%
\end{array}
\right)  \leq A_{\alpha}+\delta A^{2}_{\alpha},\label{eqt12}%
\end{equation}
where
$A_{\alpha}=D^{2}\varphi_{\alpha}({x}_{\alpha},{y}_{\alpha})\in
\mathcal{S}^2T^{\ast}(M\times M)$ and
\begin{equation}
\Vert A_{\alpha}\Vert=\sup \{|A_{\alpha}(\xi,\xi)|,\xi \in
T_{({x_{\alpha}},{y_{\alpha}})}M\times
M,|\xi|=1\}.\label{eqt12.5}%
\end{equation}
We denote
\[
P_{\alpha}=\frac{\alpha}2D^{2}\dist(x_{\alpha},y_{\alpha})^2, \text{
and
}Q_{\alpha}=\frac{\lambda_{\alpha}}2D^{2}[\dist(\hat{x}_{\alpha},x_{\alpha})^2+\dist(\hat{x}_{\alpha},y_{\alpha})^2].
\]
Hence $A_{\alpha}=P_{\alpha}+Q_{\alpha}$.  Since $\lim_{\alpha\to
+\infty}x_{\alpha}=\lim_{\alpha\to
+\infty}y_{\alpha}=\lim_{\alpha\to +\infty}\hat{x}_{\alpha}$, we can
assume $\alpha$ large enough so that all $x_{\alpha}$ and
$y_{\alpha}$ are in a small convex neighborhood $D_{\alpha}$ of
$\hat{x}_{\alpha}$. Let $\gamma$ be the minimal geodesic connecting
$x_{\alpha}$ and $y_{\alpha}$. We decompose $T_{x_{\alpha}}M\times
T_{y_{\alpha}}M$ as
$(\gamma'(0)\rr\oplus\{\gamma'(0)\}^{\bot})\times(\gamma'(0)\rr\oplus\{\gamma'(l_{\alpha})\}^{\bot})$.
According this decomposition and from Theorem \ref{thmhess}, we know
$P_{\alpha}$ satisfies
\begin{equation*}
    P_{\alpha}\le \alpha \left(
       \begin{array}{cccc}
         1&&-1&\\
         &Il_{\alpha}\kappa\coth \kappa l_{\alpha} & &-\frac{l_{\alpha}\kappa}{\sinh \kappa l_{\alpha}} I\\
         -1&&1&\\
         &-\frac{l_{\alpha}\kappa}{\sinh \kappa l_{\alpha}}I && Il_{\alpha}\kappa\coth \kappa l_{\alpha} \\
       \end{array}
     \right),
\end{equation*}
where $I$ is $(n-1)\times (n-1)$ unit matrix.
 Then
 \begin{equation*}
 \|P_{\alpha}\|\le \alpha\frac{\kappa l_{\alpha}}{\sinh \kappa
l_{\alpha}}(\cosh \kappa l_{\alpha}+1).
 \end{equation*}
We decompose $T_{x_{\alpha}}M\times T_{y_{\alpha}}M$ as $(\nabla
\dist(\hat{x}_{\alpha},\cdot)\rr\oplus\{\nabla
\dist(\hat{x}_{\alpha},\cdot)\}^{\bot})\times(\nabla
\dist(\hat{x}_{\alpha},\cdot)\rr\oplus\{\nabla
\dist(\hat{x}_{\alpha},\cdot)\}^{\bot})$. According this
decomposition, the classical Hessian comparison Theorem implies that
$Q_{\alpha}$ satisfies
\begin{equation*}
    Q_{\alpha}\le \lambda_{\alpha} \left(
       \begin{array}{cccc}
         1&&&\\
         &Il_{\alpha}\kappa\coth \kappa l_{\alpha} & &\\
         &&1&\\
         & && Il_{\alpha}\kappa\coth \kappa l_{\alpha} \\
       \end{array}
     \right).
\end{equation*}

 Since the $\dist(x_{\alpha}, y_{\alpha})$, $\dist(\hat{x}_{\alpha},x_{\alpha})$, and $\dist(\hat{x}_{\alpha},y_{\alpha})$ tend to zero as $\alpha\to +\infty$,
 we can assume the diameter of the  convex neighborhood $D_{\alpha}$ is small so that
  both $\Vert P_{\alpha}\Vert$ and $\Vert
Q_{\alpha}\Vert$ are bounded by $4\alpha$ and $2\lambda_{\alpha}$
respectively.

So for any $V\in T_{x_{\alpha}}M$ and $V\bot \gamma'$,
\begin{equation*}\begin{split}
                   \langle X_{\alpha}V,V\rangle-\langle
    Y_{\alpha}P_{\gamma}(l_{\alpha})V,P_{\gamma}(l_{\alpha})V\rangle &
    \le (P_\alpha+Q_\alpha+\delta P_\alpha^2+\delta Q_\alpha^2+\delta P_\alpha
    Q_\alpha+\delta Q_\alpha P_\alpha)(V, P_\gamma(l_\alpha)V)^2\\
                         &\le P_\alpha (V, P_\gamma(l_\alpha)V)^2+[\|Q_\alpha\|+2\delta(\|P_\alpha^2\|+\|Q_\alpha^2\|](V, P_\gamma(l_\alpha)V)^2 .
                 \end{split}
\end{equation*}
Since $\|P_\alpha\|\le 4\alpha, $  $\|Q_\alpha\|\le
2\lambda_\alpha$, we have
\begin{equation*}
    \langle X_{\alpha}V,V\rangle-\langle
    Y_{\alpha}P_{\gamma}(l_{\alpha})V,P_{\gamma}(l_{\alpha})V\rangle
    \le 2\kappa^2\alpha
    l_\alpha^2|V|^2+4\lambda_\alpha|V|^2+4\delta(16\alpha^2+4\lambda_\alpha^2)|V|^2.
\end{equation*}

For any $\varepsilon>0$, we can  choose
$\delta=\frac{\varepsilon}{64(4\alpha^2+\lambda_{\alpha}^2)}$. Since
$\alpha
    l_{\alpha}^2\to 0$ and  , then
there exists $\alpha_1>0$ such that when  $\alpha>\alpha_1$
\begin{equation*}
    \langle X_{\alpha}V,V\rangle-\langle
    Y_{\alpha}P_{\gamma}(l_{\alpha})V,P_{\gamma}(l_{\alpha})V\rangle\le
    \frac{\varepsilon}2 |V|^2.
\end{equation*}
For any $V\in T_{x_{\alpha}}M$, we have
\begin{equation*}\begin{split}
    D_x\varphi_{\alpha}(x_{\alpha},y_{\alpha})(V)&=  \alpha l_{\alpha}\langle-\gamma'(0),V\rangle+\lambda_{\alpha}\dist(\hat{x}_{\alpha},x_{\alpha})\langle\nabla \dist(\hat{x}_{\alpha},x_{\alpha}),V\rangle \\
D_y\varphi_{\alpha}(x_{\alpha},y_{\alpha})(P_{\alpha}(l_{\alpha})V)&=
 \alpha
l_{\alpha}\langle\gamma'(l_{\alpha}),P_{\alpha}(l_{\alpha})V\rangle+\lambda_{\alpha}\dist(\hat{x}_{\alpha},y_{\alpha})\langle\nabla
\dist(\hat{x}_{\alpha},y_{\alpha}),V\rangle,
                                                  \end{split}
\end{equation*}
then we have
\begin{equation*}\begin{split}
D_x\varphi_{\alpha}(x_{\alpha},y_{\alpha})(V)+&D_y\varphi_{\alpha}(x_{\alpha},y_{\alpha})(P_{\alpha}(l_{\alpha})V)=\\
&\lambda_{\alpha}\dist(\hat{x}_{\alpha},x_{\alpha})\langle\nabla
\dist(\hat{x}_{\alpha},x_{\alpha}),V\rangle+\lambda_{\alpha}\dist(\hat{x}_{\alpha},y_{\alpha})\langle\nabla
\dist(\hat{x}_{\alpha},y_{\alpha}),V\rangle,
 \end{split}
\end{equation*}
Since $\lim_{\alpha\to +\infty}\lambda_{\alpha}=0$ and
$\lim_{\alpha\to +\infty}\dist(\hat{x}_{\alpha},x_{\alpha})=0$, then
there exists $\alpha_2>0$ such that when $\alpha>\alpha_2$,
\begin{equation*}
    |D_x\varphi_{\alpha}(x_{\alpha},y_{\alpha})+D_y\varphi_{\alpha}(x_{\alpha},y_{\alpha})\circ P_{\alpha}(l_{\alpha})|<\varepsilon.
\end{equation*}
There exists  $\alpha_3>0$ such that when $\alpha>\alpha_3$,
$\omega(\sqrt{\frac{\mu_0}{\alpha}})<\varepsilon.$ Therefore
\begin{equation*}\begin{split}
                   u(x_{\alpha})-v(y_{\alpha})= & \sigma_{\alpha}+\frac{\alpha}2\dist(x_{\alpha},y_{\alpha})^2+\frac{\lambda_{\alpha}}2\dist(\hat{x}_{\alpha},x_{\alpha})^2+\frac{\lambda_{\alpha}}2\dist(\hat{x}_{\alpha},y_{\alpha})^2 \\
                     \ge&
                     \mu_0-\omega(\sqrt{\frac{\mu_0}{\alpha}})>\mu_0-\varepsilon.
                 \end{split}
    \end{equation*}
Finally we choose $\bar{\alpha}=\max\{\alpha_1,\alpha_2,\alpha_3\}$
such that all the inequalities in the theorem are satisfied. The
proof is complete.
\end{proof}

\begin{theorem}\label{thmYau}
Let $M$ be a complete Riemannian manifold with Ricci curvature
bounded below by a constant $-(n-1)\kappa^2$. Let $u
\in{\mathrm{USC}}(M)$, $v\in{\mathrm{LSC}}(M)$ be two functions
satisfying
 \begin{equation}\label{}
    \mu_0:=\sup_{x\in M}[u(x)-v(x)]< +\infty.
 \end{equation}
 Assume that $u$ and $v$ are bounded from above and below respectively and there exists a  function $\omega:\mathbb{R}%
_{+}\mapsto \mathbb{R}_{+}$ with $\omega(0)=\omega(0+)=0$ such that
\begin{equation}\label{}
    u(x)-u(y)\le \omega(\dist(x,y)).
\end{equation}
Then for each $\varepsilon>0$, there exist $x_{\varepsilon},
y_{\varepsilon}\in M$,
$(p_{\varepsilon},X_{\varepsilon})\in \bar{J}%
^{2,+}u(x_{\varepsilon}),\  \ (q_{\varepsilon},Y_{\varepsilon})\in
\bar{J} ^{2,-}v(y_{\varepsilon}),$ such that
\[
u(x_{\varepsilon})-v(y_{\varepsilon})\geq \mu_{0}-\varepsilon,\
\]
and such that
\[
\dist(x_{\varepsilon},y_{\varepsilon})<\varepsilon,\  \  \
|p_{\varepsilon }-q_{\varepsilon}\circ P_{\gamma}(l)|<\varepsilon,\
\mathrm{tr} X_{\varepsilon}\leq \mathrm{tr} Y_{\varepsilon
}+\varepsilon,\
\]
where $l=\dist(x_{\varepsilon},y_{\varepsilon})$ and $
P_{\gamma}(l)$ is the parallel transport along the shortest geodesic
connecting $x_{\varepsilon}$ and $y_{\varepsilon}$.

\end{theorem}
\begin{proof}
 Part 1 is the same as the above theorem. We start our
proof from the Part 2. Let $D_{\alpha}$ be the convex neighborhood
of $\hat{x}_{\alpha}$ chosen in the proof of the last theorem and
$-\kappa^2_{\alpha}$ is the lower bound of the sectional curvature
in $D_{\alpha}$. We can assume the diameter of $D_{\alpha}$ is small
such that both $\Vert P_{\alpha}\Vert$ and $\Vert Q_{\alpha}\Vert$
are bounded by $2\alpha$ and $2\lambda_{\alpha}$ respectively. By
Theorem 3.7, we have, for any orthonormal base $\{e_1, e_2, \cdots,
e_n\}$ at $x_{\alpha}$ with $e_1=\gamma'(0)$,
\begin{equation*}
    \sum_{i=1}^n\langle X_{\alpha}e_i, e_i\rangle-\sum_{i=1}^n\langle Y_{\alpha}P_{\gamma}(l_{\alpha})e_i,
    P_{\gamma}(l_{\alpha})e_i\rangle\le 2 (n-1)\kappa\alpha
    l_{\alpha}\frac{\sinh\frac{\kappa l_{\alpha}}2}{\cosh\frac{\kappa l_{\alpha}}2}+4(n-1)\delta(\alpha^2+\lambda_{\alpha}^2).
\end{equation*}
Here we may change the base if necessary since we are computing the
traces of $X_\alpha$ and $Y_\alpha$ respectively. Therefore we have
$\alpha_1>0,$ such that when $\alpha>\alpha_1$,
\begin{equation*}
    \mathrm{tr}X_{\alpha}- \mathrm{tr}Y_{\alpha}\le
    \frac{\varepsilon}2<\varepsilon.
\end{equation*}
The rest of the proof is the same as that of the preceding theorem.
\end{proof}
As a corollary we have the famous Yau's maximum principle.
\begin{corollary}Let $M$ be a complete Riemannian manifold with Ricci curvature
bounded below by a constant $-(n-1)\kappa^2$ and $f$ a $C^2$
 function on $M$ bounded from below. Then for any $\varepsilon>0$
 there exists a point $x_{\varepsilon}\in M$ such that
 \begin{equation*}
    f(x_{\varepsilon})\le \inf f+\varepsilon, |\nabla
    f|(x_{\varepsilon})<\varepsilon, \Delta f(x_{\varepsilon}) >-\varepsilon.
 \end{equation*}

\end{corollary}
\begin{proof}Let $u=\inf f$ and $v=f$. $\omega$ can be chosen to be a linear function.  It is straightforward to
verify that all conditions in Theorem are satisfied.
\end{proof}

\section{The maximum principle for parabolic PDE}\label{secPara}

\bigskip For any function $u:[0,T]\times M\rightarrow \mathbb{R}$, we define

\begin{definition}
We define
\[
P_{M}^{2,+}u(t_{0},x_{0})=\{(a,p,X)\in \mathbb{R}\times T_{x_{0}}^{\ast}%
M\times \mathcal{S}^2T_{x_{0}}^{\ast}M,\quad
\mathrm{satisfies}\quad(\ref{p-eq03}).\}
\]
where (\ref{p-eq03}) is%
\begin{equation}%
\begin{array}
[c]{l}%
u(t,x)\leq u(t_{0},x_{0})+a(t-t_{0})+(p(\exp_{x_{0}}^{-1}x)\\
\\
+\frac{1}{2}X(\exp_{x_{0}}^{-1}x,\exp_{x_{0}}^{-1}x)+o(|\exp_{x_{0}}%
^{-1}x|^{2}),\quad \mathrm{as}\quad x\rightarrow x_{0}, t\rightarrow
t_0.
\end{array}
\label{p-eq03}%
\end{equation}

We also set
$P_{M}^{2,-}u(t_{0},x_{0})=-P_{M}^{2,+}(-u)(t_{0},x_{0})$.

Correspondingly, we set
\[
\bar{P}_{M}^{2,+}u(t_{0},x_{0})=\left \{
\begin{split}
&  (a,p,X)\in \mathbb{R}\times T_{x_{0}}^{\ast}M\times \mathcal{S}
T_{x_{0}}^{\ast
}M,\text{ such that }(t_{0},x_{0},u(t_0, x_0),a,p,X)\text{ }\\
&  \text{is a limit point of }(t_{k},x_{k},u(t_k,
x_k),a_{k},p_{k},X_{k}), (a_{k},p_{k},X_{k})\in P_{M}
^{2,+}u(t_{k},x_{k})\text{.}
\end{split}
\right \}
\]
as well as $\bar{P}_{M}^{2,-}u(t_{0},x_{0})=-\bar{P}_{M}^{2,+}(-u)(t_{0}%
,x_{0})$.
\end{definition}

\begin{definition}
A viscosity subsolution of $\partial_{t}u+F=0$ on $(0,T)\times M$ is
a function $u\in \mathrm{USC}((0,T)\times M)$ such that
\[
a+F(x,u,p,X)\leq0\text{ }%
\]
for all $(t,x)\in(0,T)\times M$ and $(a,p,X)\in P_{M}^{2,+}u(t,x)$.
A viscosity supersolution of $\partial_{t}u+F=0$ on $(0,T)\times M$
is a function $u\in \mathrm{LSC}((0,T)\times M)$ such that
\[
a+F(x,u,p,X)\geq0\text{ }%
\]
for all $(t,x)\in(0,T)\times M$ and $(a,p,X)\in P_{M}^{2,-}u(t,x)$.
$u$ is a viscosity solution of $\partial_{t}u+F=0$ on $(0,T)\times
M$ if it is both a viscosity subsolution and a viscosity
supersolution of $\partial_{t}u+F=0$.
\end{definition}

\begin{lemma}
\label{Lm8.2}Let $u_{i}\in \mathrm{USC}((0,T)\times M_{i})$, $i=1,2$
and let
$\varphi \in C^{1,2}((0,T)\times M_{1}\times M_{2})$. Suppose that $\hat{t}%
\in(0,T)$ and $\hat{x}_{1}\in M_{1}$, $\hat{x}_{2}\in M_{2}$
satisfy:
\[
u_{1}(\hat{t},\hat{x}_{1})+u_{2}(\hat{t},\hat{x}_{2})-\varphi(\hat{t},\hat
{x}_{1},\hat{x}_{2})\geq
u_{1}(t,x_{1})+u_{2}(t,x_{2})-\varphi(t,x_{1},x_{2})
\]
for $t\in(0,T)$ and ${x}_{1}\in M_{1}$, ${x}_{2}\in M_{2}$. Assume
that there exists an $r>0$ such that for every $K>0$ there exists a
$C$ such that for
$i=1,2$: there are $b_{1},b_{2}%
\in \mathbb{R}$ and $X_{i}\in \mathcal{S}^2T_{\hat{x}_{i}}^{\ast}M_{i}$ such that%
\[%
\begin{array}
[c]{l}%
b_{i}\leq C\text{ whenever }(b_{i},q_{i},X_{i})\in P_{M_{i}}^{2,+}%
u(t,x_{i})\text{,}\\
\mathrm{d}(x_{i},\hat{x}_{i})+|t-\hat{t}|\leq r\  \text{and\ }|u_{i}%
(t,x_{i})|+|q_{i}|+\left \Vert X_{i}\right \Vert \leq K.
\end{array}
\]
Then for each $\varepsilon>0$, there exist $X_{i}\in \mathcal{S}
T_{\hat{x}}^{\ast
}M_{i}$ such that%
\[%
\begin{array}
[c]{clc}%
\mathrm{(i)} & (b_{i},D_{x_{i}}\varphi(\hat{t},\hat{x}_{1},\hat{x}_{2}%
),X_{i})\in \bar{P}_{M_{i}}^{2,+}u(\hat{t},\hat{x}_{i})\text{,} &
\text{for
}i=1,2\\
\mathrm{(ii)} & -\left(  \frac{1}{\varepsilon}+\Vert A\Vert \right)
I\leq \left(
\begin{array}
[c]{cc}%
X_{1} & 0\\
0 & X_{2}%
\end{array}
\right)  \leq A+\varepsilon A^{2}. & \\
\mathrm{(iii)} &
b_{1}+b_{2}=\partial_{t}\varphi(\hat{t},\hat{x}_{1},\hat {x}_{2}) &
\end{array}
\]
where
$A=D_{(x_{1},x_{2})}^{2}\varphi(\hat{t},\hat{x}_{1},\hat{x}_{2})$.
\end{lemma}

Now we are in a position to assert our comparison theorem for
parabolic situation.

\begin{theorem}
Let $M$ be a compact Riemannian manifold with boundary $\partial M$,
$F\in C([0,+\infty)\times \mathcal{F}(M),\mathbb{R})$ be a proper
function such that there exists a function
$\omega:[0,\infty]\rightarrow \lbrack0,\infty]$ satisfying
(\ref{eq12.7}). Let $u\in \mathrm{USC}([0,T]\times \bar{M})$ and
$v\in {\mathrm{LSC}}([0,T]\times \bar{M})$ be a subsolution and
supersolution of
\[
\partial_{t}u+F=0
\]
respectively such that $u\leq v$ on $[0,T]\times \partial M$ and
$u|_{t=0}\leq v|_{t=0}$ on $M$. Then $u\leq v$ on $[0,T]\times
\bar{M}$.
\end{theorem}

\begin{proof}
We first observe that for each $\varepsilon>0$,
$\tilde{u}=u-\varepsilon /(T-t)$  satisfies
\[
\partial_{t}\tilde{u}+F(t, x,\tilde{u}+\varepsilon /(T-t),D\tilde{u},D^{2}\tilde{u})+\frac
{\varepsilon}{(T-t)^{2}}\leq0.
\]
We thus only need to prove $\tilde{u}\leq v$. In fact it suffices to
prove the
comparison theorem for the subsolution $u$ satisfying%
\begin{equation}%
\begin{array}
[c]{l}%
\mathrm{(i)\ }\partial_{t}u+F(t, x,u+\varepsilon /(T-t),Du,D^{2}u)+\frac{\varepsilon}{(T-t)^{2}}%
\leq0\\
\mathrm{(ii)\ }\lim_{t\uparrow T}u(t,x)=-\infty,\  \
\text{uniformly on }M.
\end{array}
\label{eq8.7}%
\end{equation}
We can see that $u$, $-v$ be bounded above. For the sake of a
contradiction we assume that
\begin{equation}
\delta:=\sup_{(t,x)\in \lbrack0,T)\times M}\{u(t,x)-v(t,x)\}>0. \label{eq8.9}%
\end{equation}
Then for $\alpha$ large,
\[
0<\delta \leq \mu_{\alpha}:=\sup_{t\in \lbrack0,T),\ x,y\in M}%
\{u(t,x)-v(t,y)-\frac{\alpha}{2}{\mathrm{d}}(x,y)^{2})\}<+\infty.
\]
There exists $(t_{\alpha},x_{\alpha},y_{\alpha})$ such that
\[
\lim_{\alpha \rightarrow
\infty}[\mu_{\alpha}-(u(t_{\alpha},x_{\alpha
})-v(t_{\alpha},y_{\alpha})-\frac{\alpha}{2}{\mathrm{d}}(x_{\alpha},y_{\alpha
})^{2})]=0.
\]
Then the following holds:
\[
\left \{
\begin{array}
[c]{ll}%
(i)\quad \lim_{\alpha \rightarrow
\infty}\alpha{\mathrm{d}}(x_{\alpha},y_{\alpha
})^{2}=0,\quad \text{and} & \\
(ii)\quad \lim_{\alpha \rightarrow \infty}\mu_{\alpha}=u(\hat{t},\hat{x}%
)-v(\hat{t},\hat{x})=\delta, & \hbox{}\\
\text{ where }(\hat{t},\hat{x})=\lim_{\alpha \rightarrow
\infty}(t_{\alpha },x_{\alpha}). & \hbox{}
\end{array}
\right.
\]
Let $(t_{\alpha},x_{\alpha},y_{\alpha})$ be a maximum point of
$u(t,x)-v(t,y)-(\alpha/2){\mathrm{d}}(x_{\alpha},y_{\alpha})^{2})$
over $[0,T)\times \bar{M}\times \bar{M}$ for $\alpha>0$. Such a
maximum exists in view of the assumed bounded above on $u$, $-v$,
the compactness of $\bar{\Omega}$, and (\ref{eq8.7})(ii). The
purpose of the term $(\alpha /2)\mathrm{d}(x,y)^{2}$ is as the
elliptic case. Set
\[
M_{\alpha}=u(t_{\alpha},x_{\alpha})-v(t_{\alpha},y_{\alpha})-\frac{\alpha}%
{2}d(x_{\alpha},y_{\alpha})^{2}.
\]
By (\ref{eq8.9}), $M_{\alpha}\geq \delta$. If $t_{\alpha}=0$, we have%
\[
0<\delta \leq M_{\alpha}\leq \sup_{x,y\in \bar{M}}[u(0,x)-u(0,y)-\frac{\alpha}%
{2}d(x,y)^{2}].
\]
But the right hand side tends to zero as $\alpha \rightarrow
\infty$, so when $\alpha$ is large we have $t_\alpha>0$. Similarly,
since $u\leq v$ on $[0,T)\times \partial M$ we have $x_\alpha$,
$y_\alpha\in \Omega$.

We now apply Lemma \ref{Lm8.2} at $(t_\alpha,x_\alpha,y_\alpha)$:
there are $a,b\in \mathbb{R}$ and $X\in
\mathcal{S}^2T_{x_{\alpha}}^{\ast}M$, $Y\in \mathcal{S}
T_{y_{\alpha}}^{\ast}M$ such that%
\begin{align*}
(a,D_{x}\varphi(t_{\alpha},x_{\alpha},y_{\alpha}),X)  &  \in \bar{P}_{M}%
^{2,+}u(t_{\alpha},x_{\alpha}),\  \
(b,-D_{y}\varphi(t_{\alpha},x_{\alpha
},y_{\alpha}),Y)\in \bar{P}_{M}^{2,-}v(t_{\alpha},y_{\alpha}),\\
a  &  =b
\end{align*}
and
\[
-\left(  \frac{1}{\varepsilon}+\Vert A\Vert \right)  I\leq \left(
\begin{array}
[c]{cc}%
X & 0\\
0 & -Y
\end{array}
\right)  \leq A+\varepsilon A^{2}.
\]
The relation
\begin{align*}
a+F(t_{\alpha},x_{\alpha},u(t_{\alpha}+\varepsilon
/(T-t_\alpha),x_{\alpha}),\alpha \iota(\gamma_{\alpha
}^{\prime}(0)),X)  &  \leq-c,\\
b+F(t_{\alpha},y_{\alpha},v(t_{\alpha},y_{\alpha}),\alpha
\iota(\gamma_{\alpha }^{\prime}(l_{\alpha})),Y)  &  \geq0.
\end{align*}
We thus have
\begin{align*}
c  &
\leq-a-F(t_\alpha,x_\alpha,u(t_\alpha,x_\alpha),-l_\alpha\alpha
\iota(\gamma_{\alpha}^{\prime}(0)),X_\alpha)\\
&  \leq F(t_\alpha,y_\alpha,v(t_\alpha,y_\alpha),-l_\alpha\alpha
\iota
(\gamma_{\alpha}^{\prime}(l_{\alpha})),Y_\alpha)-F(t_\alpha,x^{\alpha
},u(t_\alpha,x_\alpha),-l_\alpha\alpha
\iota(\gamma_{\alpha}^{\prime}(0)),X_{\alpha
})\\
&  \leq \omega(\alpha
d(x_\alpha,y_\alpha)^{2}+d(x_\alpha,y_\alpha)).
\end{align*}
Let $\alpha \rightarrow \infty$ we have arrived at a contradiction.
\end{proof}

\noindent    Shige Peng\newline School of Mathematics\newline
Shandong University\newline Jinan, Shandong 250100\newline China
\newline email: peng@sdu.edu.cn

\medskip \noindent    Detang Zhou\newline Insitituto de
Matem\'atica\newline Universidade Federal Fluminense- UFF\newline
Centro, Niter\'{o}i, RJ 24020-140\newline Brazil
\newline email: zhou@impa.br
\end{document}